\documentclass{amsart}
\newtheorem{theorem}{Theorem}
\theoremstyle{plain}

\newtheorem{corollary}{Corollary}

\newtheorem{lemma}{Lemma}

\newtheorem{proposition}{Proposition}
\newtheorem{remark}{Remark}

\numberwithin{equation}{section}

\begin{document}

\noindent

\title[Gr\"{u}ss type inequality]{Some Gr\"{u}ss type inequalities for $n$-tuples of vectors in semi-inner modules}
\author[A. G. Ghazanfari, B. Ghazanfari]{A. G. Ghazanfari$^{1*}$ and B. Ghazanfari$^{2}$}

\address{$^{1,2}$ Department of Mathematics, Lorestan University, P. O. Box 465, Khoramabad, Iran.}

\email{ghazanfari.amir@gmail.com (Amir Ghasem Ghazanfari)}

\email{bahman$_{-}$ghazanfari@yahoo.com (Bahman Ghazanfari)}

\keywords{Gr\"{u}ss inequality, semi-inner product $C^*$-modules, semi-inner product $H^*$-modules.\\
\indent $^{*}$ Corresponding author.}
\begin{abstract}
Some Gr\"{u}ss type inequalities in semi-inner product modules over $C^*$-algebras and $H^*$-algebras for $n$-tuples of vectors are established.
Also we give their natural applications for the approximation of the discrete Fourier and the Melin transforms in such modules.
\end{abstract}
\maketitle

\section{Introduction}
\noindent
For two Lebesgue integrable functions $f,g:[a,b]\rightarrow \mathbb{R}$, consider the
\v{C}eby$\breve{s}$ev functional:

\begin{equation*}
T(f,g):=\frac{1}{b-a}\int_{a}^{b}f(t)g(t)dt-\frac{1}{b-a}%
\int_{a}^{b}f(t)dt\frac{1}{b-a}\int_{a}^{b}g(t)dt.
\end{equation*}%

In 1934, G. Gr\"{u}ss \cite{gru} showed that
\begin{equation}\label{1.1}
\left\vert T(f,g)\right\vert \leq \frac{1}{4}%
(M-m)(N-n),
\end{equation}%
provided $m,M,n,N$ are real numbers with the property $-\infty <m\leq f\leq
M<\infty $ and $-\infty <n\leq g\leq N<\infty \quad \text{a.e. on }[a,b].$
The constant $\frac{1}{4}$ is best possible in the sense that it cannot be
replaced by a smaller quantity and is achieved for
\[
f(x)=g(x)=sgn \Big(x-\frac{a+b}{2}\Big).
\]
The discrete version of (\ref{1.1}) states that:
If $a\leq a_i\leq A,~b\leq b_i\leq B,~(i=1,...,n)$ where $a,A,b,B,a_i,b_i$ are real
numbers, then
\begin{equation}\label{1.2}
\left| \frac{1}{n}\sum_{i=1}^na_ib_i-\frac{1}{n}\sum_{i=1}^na_i.\frac{1}{n}\sum_{i=1}^nb_i\right|\leq \frac{1}{4}(A-a)(B-b),
\end{equation}
where the constant $\frac{1}{4}$ is the best possible for an arbitrary $n\geq 1$. Some refinements of the
discrete version of Gr\"{u}ss inequality (\ref{1.2}) are given in \cite{kech}.

In the recent years, this inequality has been investigated, applied
and generalized by many authors in different areas of mathematics, among others
in inner product spaces \cite{dra1},
in the approximation of integral transforms \cite{li} and the references therein,
 in semi-inner $*$-modules for positive linear functionals
and $C^*$-seminorms \cite{gha}, for positive maps \cite{mos1}.

A good example of how Gr\"{u}ss type inequalities
can cross mathematical categories is provided by the development of the Gr\"{u}ss type inequalities in
inner product modules over $H^*$-algebras and $C^*$-algebras \cite{gha1,ili}.
For an entire chapter devoted to the history of this inequality see \cite{mit}
where further references are given.

We recall some of the most important Gr\"{u}ss type discrete
inequalities for inner product spaces that are available in \cite{dra}.

\begin{theorem}\label{t1.1} Let $(H; \langle \cdot, \cdot\rangle)$ be an inner product space over $\mathbb{K};~
(\mathbb{K} = \mathbb{C},\mathbb{R}),~ x_i,~ y_i \in H,~ p_i \geq 0 ~(i = 1, . . . , n)~ (n \geq 2)$ with
$\sum_{i=1}^n p_i= 1$. If $x, X, y, Y\in H$ are such that
\begin{equation*}
 Re \left<X- x_i, x_i - x \right>\geq 0\quad and\quad Re \left<Y - y_i, y_i -y\right>\geq 0
\end{equation*}
for all $i\in \{1, . . . , n\}$, or, equivalently,
\begin{equation*}
\left\|x_i-\frac{x+X}{2}\right\|\leq\frac{1}{2}\|X-x\| \quad and \quad \left\|y_i-\frac{y+Y}{2}\right\|
\leq\frac{1}{2}\|Y-y\|
\end{equation*}
for all $i\in \{1, . . . , n\}$, then the following inequality holds
\begin{equation*}
\left|\sum_{i=1}^n p_i \langle x_i ,y_i \rangle-\left< \sum_{i=1}^n p_i x_i, \sum_{i=1}^n p_i y_i \right>\right|
\leq\frac{1}{4}\|X-x\|\|Y-y\|.
\end{equation*}
The constant $\frac{1}{4}$ is best possible in
the sense that it cannot be replaced by a smaller quantity.
\end{theorem}

\begin{theorem}
\label{t1.2} Let $(H; \langle \cdot, \cdot\rangle)$ and $\mathbb{K}$ be as above and
$\overline{x}=(x_1, . . . , x_n)\in H^n$, $\overline{\alpha}=(\alpha_1, . . . , \alpha_n) \in \mathbb{K}^n$ and
$\overline{p}=(p_1,...,p_n)$ a probability vector. If $x,X \in H$ are
such that
\begin{equation*}
 Re \left<X- x_i, x_i - x \right>\geq 0~for~ all~ i\in \{1, . . . , n\},
\end{equation*}
or, equivalently,
\begin{equation*}
\left\|x_i-\frac{x+X}{2}\right\|\leq\frac{1}{2}\|X-x\| ~for~ all~ i\in \{1, . . . , n\},
\end{equation*}
holds, then the following inequality holds
\begin{align*}
\left\|\sum_{i=1}^np_i\alpha_i x_i-\sum_{i=1}^n p_i \alpha_i \sum_{i=1}^n p_ix_i\right\|
&\leq\frac{1}{2}\|X-x\|\sum_{i=1}^np_i\left|\alpha_i-\sum_{j=1}^np_j\alpha_j\right|\\
&\leq\frac{1}{2}\|X-x\|\left[\sum_{i=1}^np_i|\alpha_i|^2-\left|\sum_{i=1}^np_i\alpha_i\right|^2\right]^{\frac{1}{2}}.
\end{align*}
The constant $\frac{1}{2}$ in the first and second inequalities is best possible.
\end{theorem}

Motivated by the above results we establish some new Gr\"{u}ss type inequalities in semi-inner product modules over $C^*$-algebras
and $H^*$-algebras for $n$-tuples of vectors, which are generalizations of Theorem \ref{t1.1} and Theorem \ref{t1.2}.
We also give some their applications for the approximation of the discrete Fourier and Melin transforms.
In order to do that we need the following preliminary definitions and results.

\section{Preliminaries}

Hilbert $C^*$-modules are used as the framework for Kasparov's bivariant K-theory and form the
technical underpinning for the $C^*$-algebraic approach to quantum groups. Hilbert $C^*$-modules
are very useful in the following research areas: operator K-theory, index theory for operator-valued
conditional expectations, group representation theory, the theory of $AW^*$-algebras, noncommutative
geometry, and others. Hilbert $C^*$-modules form a category in between Banach
spaces and Hilbert spaces and obey the same axioms as a Hilbert space except that the inner
product takes values in a general $C^*$-algebra rather than in the complex number $\mathbb{C}$. This simple generalization
gives a lot of trouble. Fundamental and familiar Hilbert space properties like Pythagoras' equality, self-duality and
decomposition into orthogonal complements must be given up. Moreover,
a bounded module map between Hilbert $C^*$-modules does not need to have an adjoint; not every
adjointable operator needs to have a polar decomposition. Hence to get its applications, we have to
use it with great care.

A proper $H^*$-algebra is a complex Banach $*$-algebra $(\mathcal{A}, \|.\|)$ where the underlying
Banach space is a Hilbert space with respect to the inner product $\langle .,.\rangle$ satisfying the properties
$\langle ab,c\rangle=\langle b,a^*c\rangle=\langle a,cb^*\rangle$ for all
$a, b, c \in\mathcal{A}$. A $C^*$-algebra is a complex
Banach $*$-algebra $(\mathcal{A}, \|.\|)$ such that $\|a^*a\|=\|a\|^2$ for every $a \in \mathcal{A}$.
If $\mathcal{A}$ is a proper $H^*$-algebra or a $C^*$-algebra and $a\in \mathcal{A}$ is such that $\mathcal{A}a = 0$ or
$a\mathcal{A} =0$ then $a=0$.
An element $a$ in a proper $H^*$-algebra $\mathcal{A}$ is called positive $(a\geq 0)$ if $\langle ax,x\rangle\geq 0$
for every $x\in \mathcal{A}$. Every positive element $a$ in a proper $H^*$-algebra is self-adjoint
(that is $a^* = a$). An element $a$ in a $C^*$-algebra $\mathcal{A}$ is called positive $(a\geq 0)$ if it is
self-adjoint and has positive spectrum. An element $a^*a$ is positive for every $a\in \mathcal{A}$,
in both structures.

For a proper $H^*$-algebra $\mathcal{A}$, the trace class associated with $\mathcal{A}$ is
$\tau(\mathcal{A}) = \{ab : a, b \in \mathcal{A}\}$. It is a self-adjoint two-sided ideal of $\mathcal{A}$ which is dense in $\mathcal{A}$.
For every positive $a\in \tau(\mathcal{A})$ there exists the square root of $a$,
that is, a unique positive $a^{\frac{1}{2}} \in \mathcal{A}$ such that $\big(a^{\frac{1}{2}}\big)^2 = a$, the square root of $a^*a$ is
denoted by $|a|$. There are a positive linear
functional $tr$ on $\tau(\mathcal{A})$ and a norm $\tau$ on $\tau(\mathcal{A})$, related to the norm of $\mathcal{A}$ by the equality
$tr(a^*a) = \tau (a^*a) = \|a\|^2 $ for every $a \in \mathcal{A}$.
The trace-class is a Banach $*$-algebra with respect to the norm $\tau(.)$ defined
by $\tau(a) = tr(|a|)$. Let us mention that $|tr(a)|\leq  \tau(a)$ and $\|a\|\leq \tau(a)$ for every
$a\in \tau(\mathcal{A})$.

Let $\mathcal{A}$ be a proper $H^*$-algebra or a $C^*$-algebra. A semi-inner product module
over $\mathcal{A}$ is a right module $X$ over $\mathcal{A}$ together with a generalized semi-inner product,
that is, with a mapping $\langle.,.\rangle$ on $X\times X$, which is $\tau(\mathcal{A})$-valued if $\mathcal{A}$ is a proper
$H^*$-algebra, or $\mathcal{A}$-valued if $\mathcal{A}$ is a $C^*$-algebra, has the following properties:
\begin{enumerate}
\item[(i)]
 $\langle x,y+z\rangle=\langle x,y\rangle+\langle x,z\rangle$ for all $x,y,z\in X,$
\item[(ii)]
$\left\langle x,ya\right\rangle=\left\langle x,y\right\rangle a$ for $x,y\in X, a\in \mathcal{A}$,
\item[(iii)]
$\langle x,y\rangle ^*=\langle y,x\rangle$ for all $x,y\in X$,
\item[(iv)]
$\left\langle x,x\right\rangle \geq 0$ for $x\in X$.
\end{enumerate}
We will say that $X$ is a semi-inner product $H^*$-module if $\mathcal{A}$ is a proper $H^*$-algebra and
that $X$ is a semi-inner product $C^*$-module if $\mathcal{A}$ is a $C^*$-algebra.\\
The absolute value of $x\in X$ is defined as the square root of $\langle x,x\rangle$, and it is denoted
by $|x|$.

If, in addition,
\begin{enumerate}
\item[(v)]
$\langle x,x\rangle=0$ implies $x=0$,
\end{enumerate}
then $\langle .,.\rangle$ is called a generalized inner
product and $X$ is called an inner product module over $\mathcal{A}$.
We will say that $X$ is a (semi-)inner product $H^*$-module if it is a (semi-)inner product module over a proper
$H^*$-algebra, and that $X$ is a (semi-)inner product $C^*$-module if it is a (semi-)inner
product module over a $C^*$-algebra.

As we can see, an inner product module obeys the same axioms as an ordinary
inner product space, except that the inner product takes values in a more general
structure rather than in the field of complex numbers.

%---------------------------------------------------------------------------------------%
If $\mathcal{A}$ is a $C^*$-algebra and $X$ is a semi-inner product $\mathcal{A}$-module, then the following Schwarz
inequality holds:
\begin{equation}
\|\langle x,y\rangle\|^2\leq \|\langle x,x\rangle\|\|\langle y,y\rangle\|~(x,y\in X)\label{2.1}
\end {equation}
($e.g.$ \cite[Lemma 15.1.3]{weg}).

It follows from the Schwarz inequality (\ref{2.1}) that $\|x\|:=\|\langle x,x\rangle\|^\frac{1}{2}\quad (x\in X)$, is a semi-norm on $X$.

If $X$ is a semi-inner product $H^*$-module, then there are two forms of the Schwarz inequality: for every $x,y\in X$
\begin{align}
& (tr\langle x,y\rangle)^2\leq tr\langle x,x\rangle tr\langle y,y\rangle  &&\text{(the weak Schwarz inequality)},\label{2.2}\\
&(\tau\langle x,y\rangle)^2\leq tr\langle x,x\rangle tr\langle y,y\rangle &&\text{(the strong Schwarz inequality)}. \label{2.3}
\end{align}
First Saworotnow in \cite{saw} proved the strong Schwarz inequality, but the direct proof of that for a semi-inner product $H^*$-module
can be found in \cite{mol}.

Weak Schwarz inequality (\ref{2.2}) implies that $\big\| |x|\big\|=\big(tr\langle x,x\rangle\big)^\frac{1}{2}~(x\in X)$,
 is a semi-norm on $X$.

Now let $\mathcal{A}$ be a $\ast $-algebra, $\varphi $ a positive linear functional on $%
\mathcal{A}$, and let $X$ be a semi-inner $\mathcal{A}$-module. We can define a
sesquilinear form on $X\times X$ by $\sigma (x,y)=\varphi \left(
\left\langle x,y\right\rangle \right) $; the Schwarz inequality for $\sigma $
implies that
\begin{equation}
\vert \varphi\langle x,y\rangle\vert^{2}\leq \varphi\langle x,x\rangle \varphi\langle y,y\rangle.  \label{2.4}
\end{equation}%
In \cite[Proposition 1, Remark 1]{gha} the authors present two other forms of the Schwarz inequality in
semi-inner $\mathcal{A}$-module $X$, one for
a positive linear functional $\varphi$ on $\mathcal{A}$:
\begin{equation}
\varphi(\langle x,y\rangle\langle y,x\rangle)
\leq \varphi \langle x,x\rangle r\langle y,y\rangle,  \label{2.5}
\end{equation}%
where $r$ is the spectral radius, and another one for a $C^*$-seminorm $\gamma$ on $\mathcal{A}$:
\begin{equation}
(\gamma\langle x,y\rangle) ^{2}\leq \gamma\langle x,x\rangle \gamma \langle y,y\rangle.  \label{2.6}
\end{equation}%

%-------------------------------------------------------------------------------------------------------------------------------------------------%
Before stating the main results, let us fix the rest of our notation. We assume that $\mathcal{A}$ is a $C^\ast $-algebra or a $H^\ast $-algebra, and assume unless stated otherwise,
throughout this paper $\overline{p}=(p_1,...,p_n)\in \mathbb{R}^n$ a
probability vector i.e.
$p_i \geq 0 \quad (i = 1, . . . , n)$ and $\sum_{i=1}^n p_i=1 $. If $X$ is a semi-inner product $\mathcal{A}$-module and
$\overline{x}=(x_1,...,x_n), \overline{y}=(y_1,...,y_n)\in X^n$ we put
$$G_{\overline{p}}(\overline{x},\overline{y})=\sum_{i=1}^n p_i \langle x_i ,y_i \rangle-
\left< \sum_{i=1}^n p_i x_i, \sum_{i=1}^n p_i y_i \right>,$$
we use $G_{\overline{p}}(\overline{x})$ instead of $G_{\overline{p}}(\overline{x},\overline{x})$.

\begin{lemma}\label{l2.1}
Let $X$ be a semi-inner product $C^*$-module or a semi-inner $H^\ast $-module, $a,b\in X$, $\overline{x}=(x_1,...,x_n), \overline{y}=(y_1,...,y_n)\in X^n, \overline{\alpha}=(\alpha_1, . . . , \alpha_n)
 \in \mathbb{K}^n$; $(\mathbb{K} = \mathbb{C},\mathbb{R})$ and $\overline{p}=(p_1,...,p_n)\in \mathbb{R}^n$ a probability vector, then
\begin{equation}
\sum_{i=1}^np_i\alpha_i x_i-\sum_{i=1}^n p_i \alpha_i \sum_{i=1}^n p_ix_i=\sum_{i=1}^np_i\Big(\alpha_i-\sum_{j=1}^np_j\alpha_j\Big)
(x_i-a),\label{2.7}
\end{equation}
and
\begin{equation}
G_{\overline{p}}(\overline{x},\overline{y})=
\sum_{i=1}^n p_i\left<x_i -a, y_i -b \right>-\left< \sum_{i=1}^n p_i (x_i -a), \sum_{i=1}^n p_i (y_i-b)\right>.\label{2.8}
\end{equation}
In particular
\begin{equation}
G_{\overline{p}}(\overline{x})=\sum_{i=1}^n p_i\left|x_i -a\right|^2-\left| \sum_{i=1}^n p_i x_i -a\right|^2
\leq\sum_{i=1}^n p_i\left|x_i -a\right|^2.\label{2.9}
\end{equation}
\end{lemma}

\begin{proof}
For every $a\in X$ a simple calculation shows that
\begin{multline}
\sum_{i=1}^np_i\Big(\alpha_i-\sum_{j=1}^np_j\alpha_j\Big)(x_i-a)=\sum_{i=1}^np_i\alpha_i x_i-\sum_{j=1}^n p_j \alpha_j\sum_{i=1}^np_ix_i\notag\\
-a\sum_{i=1}^np_i\alpha_i+a\sum_{i=1}^np_i\sum_{j=1}^n p_j \alpha_j\\
=\sum_{i=1}^np_i\alpha_i x_i-\sum_{i=1}^n p_i \alpha_i \sum_{i=1}^n p_ix_i.
\end{multline}

For every $a,b\in X$, a simple calculation shows that

\begin{multline}
\sum_{i=1}^n p_i\left<x_i -a, y_i -b \right>-\left< \sum_{i=1}^n p_i (x_i -a), \sum_{i=1}^n p_i (y_i-b)\right>\notag\\
=\sum_{i=1}^n p_i\big(\langle x_i,y_i\rangle-\langle x_i,b\rangle-\langle a,y_i\rangle+\langle a,b\rangle\big)\\
-\left< \sum_{i=1}^n p_ix_i -a, \sum_{i=1}^n p_iy_i-b\right>\\
=\sum_{i=1}^n p_i \langle x_i ,y_i \rangle-\left< \sum_{i=1}^n p_i x_i, \sum_{i=1}^n p_i y_i \right>=
G_{\overline{p}}(\overline{x},\overline{y}).
\end{multline}
In particular for $a=b, x_i=y_i$ we have
\begin{multline}
G_{\overline{p}}(\overline{x})=\sum_{i=1}^n p_i\left<x_i -a, x_i -a \right>-\left< \sum_{i=1}^n p_i (x_i -a), \sum_{i=1}^n p_i (x_i-a)\right>\notag\\
=\sum_{i=1}^n p_i\left|x_i -a\right|^2-\left| \sum_{i=1}^n p_i (x_i -a)\right|^2\leq
\sum_{i=1}^n p_i\left|x_i -a\right|^2\\
\end{multline}

\end{proof}
%-------------------------------------------------------------------------------------------------------------------------------------------------%

\section{Gr\"{u}ss type inequalities in semi-inner product $C^*$-modules}

In the following Theorem we give a generalization of Theorem \ref{t1.1} for semi-inner product
$C^*$-modules.

\begin{theorem}
\label{t3.1} Let $X$ be a semi-inner product $C^*$-module, $a,b\in X$ and $\overline{p}=(p_1,...,p_n)\in \mathbb{R}^n$ a probability vector. If $\overline{x}=(x_1,...,x_n), \overline{y}=(y_1,...,y_n)\in X^n$,
then the following inequality holds
\begin{multline}
\left\|\sum_{i=1}^n p_i\left<x_i, y_i \right>-\left< \sum_{i=1}^n p_i x_i, \sum_{i=1}^n p_i y_i\right>\right\|^2\\
\leq \left\|\sum_{i=1}^n p_i\left|x_i -a\right|^2-\left| \sum_{i=1}^n p_i x_i -a\right|^2\right\|\label{3.1}
\left\|\sum_{i=1}^n p_i\left|y_i -b\right|^2-\left| \sum_{i=1}^n p_i y_i -b\right|^2\right\|\\
\leq\left(\sum_{i=1}^n p_i\left\|x_i -a\right\|^2\right)\left(\sum_{i=1}^n p_i\left\|y_i -b\right\|^2\right)
\end{multline}
\end{theorem}

\begin{proof}
A simple calculation shows that
\begin{equation}
\sum_{i=1}^n p_i \langle x_i ,y_i \rangle-\left< \sum_{i=1}^n p_i x_i, \sum_{i=1}^n p_i y_i \right>=\frac{1}{2}
\sum_{i,j=1}^n p_ip_j\left<x_i -x_j ,y_i -y_j\right>,\label{3.2}
\end{equation}
therefore
\begin{equation}
G_{\overline{p}}(\overline{x})=\frac{1}{2}\sum_{i,j=1}^n p_ip_j\left<x_i -x_j ,x_i -x_j\right>\geq 0.\label{3.3}
\end{equation}
It is easy to show that $G_{\overline{p}}(\cdot,\cdot)$ is an $\mathcal{A}$-value semi-inner product on $X^n$,
so Schwarz inequality holds i.e.,
\begin{equation}
\|G_{\overline{p}}(\overline{x},\overline{y})\|^2\leq \|G_{\overline{p}}(\overline{x})\|
\|G_{\overline{p}}(\overline{y})\|.\label{3.4}
\end{equation}
From inequality (\ref{2.9}) we get
\begin{equation}
\|G_{\overline{p}}(\overline{x})\|=\left\|\sum_{i=1}^n p_i\left|x_i -a\right|^2-\left| \sum_{i=1}^n p_i x_i -a\right|^2\right\|
\leq \sum_{i=1}^n p_i\left\|x_i -a\right\|^2.\label{3.5}
\end{equation}
similarly
\begin{equation}
\|G_{\overline{p}}(\overline{y})\|=\left\|\sum_{i=1}^n p_i\left|y_i -b\right|^2-\left| \sum_{i=1}^n p_i y_i -b\right|^2\right\|
\leq \sum_{i=1}^n p_i\left\|y_i -b\right\|^2.\label{3.6}
\end{equation}
Therefore we obtain the inequality (\ref{3.1}).
\end{proof}

\begin{corollary}
\label{c3.1} Let $X$ be a semi-inner product $C^*$-module, $a,b\in X$ and $\overline{p}=(p_1,...,p_n)\in \mathbb{R}^n$ a probability vector. If $\overline{x}=(x_1,...,x_n), \overline{y}=(y_1,...,y_n)\in X^n$, $r\geq 0, s\geq 0$ are such that
\begin{equation}
\|x_i-a\|\leq r,\quad \|y_i-b\|\leq s,\text{ for all i }\in \{1, . . . , n\},\label{3.7}
\end{equation}
then the following inequality holds
\begin{equation}
\left\|\sum_{i=1}^n p_i\left<x_i, y_i \right>-\left< \sum_{i=1}^n p_i x_i, \sum_{i=1}^n p_i y_i\right>\right\|\leq rs.\label{3.8}
\end{equation}
The constant 1 coefficient of $rs$ in the inequality (\ref{3.8}) is best possible in
the sense that it cannot be replaced by a smaller quantity.
\end{corollary}

\begin{proof}
From inequalities (\ref{3.1}) and (\ref{3.7}) we obtain (\ref{3.8}).

To prove the sharpness of the constant 1 in the inequality in (\ref{3.8}), let us assume that, under the assumptions of the theorem, the
inequalities hold with a constant $c>0$, i.e.,

\begin{equation}
\|G_{\overline{p}}(\overline{x},\overline{y})\|\leq crs \label{3.9}.
 \end{equation}
Assume that $n = 2, p_1 = p_2 = \frac{1}{2}$ and $e$ is an element of $X$ such that $\|\langle e,e\rangle\|=1$.
We put
\begin{align*}
x_1&=a+re,~~~y_1=b+se\\
x_2&=a-re,~~~y_2=b-se,
\end{align*}
then, obviously,
\begin{equation*}
\|x_i-a\|\leq r,\quad \|y_i-b\|\leq s,\quad (i=1,2),
\end{equation*}
which shows that the condition (\ref{3.7}) holds. If we replace $n, p_1, p_2, x_1, x_2, y_1, y_2$ in (\ref{3.9}), we obtain
\begin{equation*}
\|G_{\overline{p}}(\overline{x},\overline{y})\|=rs\leq crs,
\end{equation*}
from where we deduce that $c\geq 1$, which proves the sharpness of the constant 1.
\end{proof}

The following Remark \ref{r3.1}(ii) is a generalization of Theorem \ref{t1.2} for semi-inner product $C^*$-modules.

\begin{remark}\label{r3.1}\rm{
\begin{enumerate}
\item[(i)]
Let $\mathcal{A}$ be a $C^{\ast}$-algebra, and $\overline{p}=(p_1,...,p_n)\in \mathbb{R}^n$ a probability vector. If $a, b, a_i, b_i, (i=1,2,...,n) \in \mathcal{A}, r\geq 0, s\geq0$ are such that
\begin{equation*}
\|a_i-a\|\leq r,~~\|b_i-b\|\leq s, \text{ for all i }\in \{1, . . . , n\},
\end{equation*}
it is known that $\mathcal{A}$ is a Hilbert $C^*$-module over
itself with the inner product defined by $\left\langle a,b\right\rangle :=a^{\ast }b$.
In this case (\ref{3.8}) implies that
\begin{align*}
\left\|\sum_{i=1}^n p_ia_i^*b_i-\sum_{i=1}^n p_ia_i^*.\sum_{i=1}^n p_ib_i\right\|\leq rs.
\end{align*}
Since
\begin{equation*}
\|a_i^*-a^*\|\leq r,~~ \text{ for all i }\in \{1, . . . , n\},
\end{equation*}
we deduce
\begin{align*}
\left\|\sum_{i=1}^n p_ia_ib_i-\sum_{i=1}^n p_ia_i.\sum_{i=1}^n p_ib_i\right\|\leq rs.
\end{align*}

\item[(ii)] Let $X$ be a semi-inner product $C^{\ast}$-module, $a\in X, \overline{\alpha}=(\alpha_1, . . . , \alpha_n)
 \in \mathbb{K}^n$ and $\overline{p}=(p_1,...,p_n)\in \mathbb{R}^n$ a probability vector. If $\overline{x}=(x_1,...,x_n)\in X^n, r\geq 0$ are such that
\begin{equation*}
\|x_i-a\|\leq r,\text{ for all i }\in \{1, . . . , n\},
\end{equation*}
holds, from equality (\ref{2.7}) we obtain
\begin{align}
\left\|\sum_{i=1}^np_i\alpha_i x_i-\sum_{i=1}^n p_i \alpha_i \sum_{i=1}^n p_ix_i\right\|
&\leq r\sum_{i=1}^np_i\left|\alpha_i-\sum_{i=1}^np_j\alpha_j\right|\label{3.10}\\
&\leq r\left[\sum_{i=1}^np_i|\alpha_i|^2-\left|\sum_{i=1}^np_i\alpha_i\right|^2\right]^{\frac{1}{2}}.\notag
\end{align}
The constant 1 in the first and second inequalities in (\ref{3.10}) is best possible.

\end{enumerate}}
\end{remark}

\section{Applications}
In this section we give applications of Corollary \ref{c3.1} for the approximation of some discrete transforms such as the discrete Fourier and the Melin transforms.
Let $X$ be a semi-inner product $C^*$-module on $C^*$-algebra $\mathcal{A}$ and  $x=(x_1,...,x_n), y=(y_1,...,y_n)\in X^n$. For a given $\omega\in \mathbb{R}$,
define the \textit{discrete Fourier transform}

\begin{equation}
\mathcal{F}_\omega(x)(m)=\sum_{k=1}^n \exp(2\omega imk)\times x_k,\quad m=1,...,n.\label{4.1}
\end{equation}
The element $\sum_{k=1}^n \exp(2\omega imk)\times \langle x_k, y_k\rangle$ of $\mathcal{A}$ is called
Fourier transform of the vector $(\langle x_1,y_1\rangle,...,\langle x_k,y_k\rangle)\in \mathcal{A}^n$
and will be denoted by
\begin{equation}
\mathcal{F}_\omega(x,y)(m)=\sum_{k=1}^n \exp(2\omega imk)\times \langle x_k,y_k\rangle\quad m=1,...,n.\label{4.2}
\end{equation}
The following Theorems \ref{t4.1}, \ref{t4.2} and \ref{t4.3} are generalizations of \cite[Theorems 66, 67, 68]{dra} for semi-inner product $C^*$-modules respectively.

\begin{theorem}
\label{t4.1} Let $X$ be a semi-inner product $C^*$-module, $a, b \in X$. If $x=(x_1,...,x_n), y=(y_1,...,y_n)\in X^n, r\geq 0, s\geq 0$ are such that
\begin{equation}
\|x_k-a\|\leq r,\quad \big\|\exp (2\omega imk)y_k-b\big\|\leq s,\text{ for all k,m }\in \{1, . . . , n\},\label{4.3}
\end{equation}
then the following inequality holds
\begin{equation}
\left\|\mathcal{F}_\omega(x,y)(m)-\left\langle \frac{1}{n}\sum_{k=1}^n x_k, \mathcal{F}_\omega(y)(m)\right\rangle\right\|\leq nrs,\label{4.4}
\end{equation}
for all $m\in\{1,...,n\}$.
\end{theorem}

\begin{proof}
By Corollary \ref{c3.1} applied for $p_k =\frac{1}{n}$ and for the
vectors $x_k$ and $\exp (2\omega imk)y_k (k = 1, . . . , n)$, we get
\begin{multline*}
\left\|\mathcal{F}_\omega(x,y)(m)-\left\langle \frac{1}{n}\sum_{k=1}^n x_k, \mathcal{F}_\omega(y)(m)\right\rangle\right\|\\
=n\left\|\sum_{k=1}^n\frac{1}{n}\Big\langle x_k, \exp (2\omega imk)y_k\Big\rangle-\left\langle\sum_{k=1}^n\frac{1}{n}x_k,\sum_{k=1}^n\frac{1}{n}\exp (2\omega imk)y_k\right\rangle\right\|\leq nrs.
\end{multline*}
\end{proof}

We can also consider the \textit{Mellin transform}
\begin{equation}
\mathcal{M}(x)(m)=\sum_{k=1}^n k^{m-1}x_k, \quad m=1,...,n.\label{4.5}
\end{equation}
of the vector $x=(x_1,...,x_n) \in X^n$.\\
The Mellin transform of the vector $(\langle x_1,y_1\rangle,...,\langle x_k,y_k\rangle)\in \mathcal{A}^n$ is defined by\\
$\sum_{k=1}^n k^{m-1}\langle x_k,y_k\rangle$ and will
be denoted by
\begin{equation}
\mathcal{M}(x,y)(m)=\sum_{k=1}^n k^{m-1}\langle x_k,y_k\rangle.\label{4.6}
\end{equation}

\begin{theorem}
\label{t4.2} Let $X$ be a semi-inner product $C^*$-module, $a,b\in X$. If $x=(x_1,...,x_n), y=(y_1,...,y_n)\in X^n, r\geq 0, s\geq 0$ are such that
\begin{equation}
\|x_k-a\|\leq r,\quad \left\|k^{m-1}y_k-b\right\|\leq s,\text{ for all k,m }\in \{1, . . . , n\},\label{4.7}
\end{equation}
then the inequality
\begin{equation}
\left\|\mathcal{M}(x,y)(m)-\left\langle \frac{1}{n}\sum_{k=1}^n x_k, \mathcal{M}(y)(m)\right\rangle\right\|\leq nrs,\label{4.8}
\end{equation}
holds for all $m\in\{1, . . . , n\}$.
\end{theorem}
The proof follows by Corollary \ref{c3.1} applied for $p_k =\frac{1}{n}$ and for the
vectors $x_k $ and $k^{m-1}y_k (k = 1, . . . , n)$. We omit the details.

Another result which connects the Fourier transforms for different
parameters $\omega$ also holds.

\begin{theorem}\label{t4.3}
 Let $X$ be a semi-inner product $C^*$-module, $a,b\in X$. If $x=(x_1,...,x_n), y=(y_1,...,y_n)\in X^n, r\geq 0, s\geq 0$ are such that
for all $k,m \in \{1, . . . , n\}$,
\begin{equation}
\big\|\exp (2\omega_1 imk)x_k-a\big\|\leq r,~ \big\|\exp (2\omega_2 imk)y_k-b\big\|\leq s,\label{4.9}
\end{equation}
then for all $m\in\{1,...,n\}$, the following inequality holds
\begin{equation}
\left\|\frac{1}{n}\mathcal{F}_{\omega_2-\omega_1}(x,y)(m)-\left\langle \frac{1}{n}\mathcal{F}_{\omega_1}(x)(m), \frac{1}{n}\mathcal{F}_{\omega_2}(y)(m)\right\rangle\right\|\leq rs.\label{4.10}
\end{equation}

\end{theorem}
The proof follows by Theorem \ref{t3.1} applied for $p_k =\frac{1}{n}$ and for the
vectors $ \exp (2\omega_1imk) x_k$ and $\exp(2\omega_2imk)y_k (k = 1, . . . , n)$. We omit the details.

Let $X$ be a semi-inner product $C^{\ast}$-module, $x=(x_1,...,x_n), \alpha=(\alpha_1, . . . , \alpha_n)
 \in \mathbb{C}^n$ and $p=(p_1,...,p_n)\in \mathbb{R}^n$ a probability vector. If $a\in X, r\geq 0,$ such that
\begin{equation*}
\|x_i-a\|\leq r,\text{ for all i }\in \{1, . . . , n\},
\end{equation*}
holds, from equality (\ref{2.7}) in Lemma \ref{l2.1} we get
\begin{align}
\left\|\sum_{i=1}^np_i\alpha_i x_i-\sum_{i=1}^n p_i \alpha_i \sum_{i=1}^n p_ix_i\right\|
&\leq r\sum_{i=1}^np_i\left|\alpha_i-\sum_{i=1}^np_j\alpha_j\right|\label{4.11}\\
&\leq r\left[\sum_{i=1}^np_i|\alpha_i|^2-\left|\sum_{i=1}^np_i\alpha_i\right|^2\right]^{\frac{1}{2}}.\notag
\end{align}
The constant 1 coefficient of $r$ in the first and second inequalities in (\ref{4.11}) is best possible.

The following approximation result for the Fourier transform (\ref{4.1}) which is a generalization of \cite[Theorem 3]{dra1}
in semi-inner product $C^*$-modules holds.

\begin{proposition}
Let $X$ be a semi-inner product $C^*$-module, $a\in X$. If $x=(x_1,...,x_n)\in X^n, r\geq 0$ are such that
\begin{equation*}
\|x_i-a\|\leq r,\text{ for all i }\in \{1, . . . , n\},
\end{equation*}
then for all $m\in\{1,...,n\}$ and $\omega\in \mathbb{R}, \omega\neq\frac{l}{m}\pi, l\in\mathbb{Z}$
 the following inequality holds
\begin{multline}
\left\|\mathcal{F}_\omega(x)(m)-\frac{\sin(\omega mn)}{\sin(\omega m)}\exp[\omega(n+1)im]\times \frac{1}{n}\sum_{k=1}^n x_k\right\|\label{4.12}\\
\leq r\left[n^2-\frac{\sin^2(\omega mn)}{\sin^2(\omega m)}\right]^{\frac{1}{2}}.
\end{multline}
\end{proposition}

\begin{proof}
From the inequality (\ref{4.11}) we can state that,
\begin{equation}
\left\|\frac{1}{n}\sum_{i=1}^n\alpha_i x_i-\frac{1}{n}\sum_{i=1}^n \alpha_i .\frac{1}{n}\sum_{i=1}^n x_i\right\|
\leq r\left[\frac{1}{n}\sum_{i=1}^n|\alpha_i|^2-\left|\frac{1}{n}\sum_{i=1}^n\alpha_i\right|^2\right]^{\frac{1}{2}},\label{4.13}
\end{equation}
for all $\alpha_i \in \mathbb{C},~ x_i\in X ~(i=1,...,n)$.
Consequently, we conclude that
\begin{equation}
\left\|\sum_{i=1}^n\alpha_i x_i-\sum_{i=1}^n \alpha_i .\frac{1}{n}\sum_{i=1}^n x_i\right\|
\leq r\left[n\sum_{i=1}^n|\alpha_i|^2-\left|\sum_{i=1}^n\alpha_i\right|^2\right]^{\frac{1}{2}}.\label{4.14}
\end{equation}
A simple calculation shows that (see the proof of Theorem 59 in \cite{dra}),
\begin{equation*}
\sum_{k=1}^n \exp(2\omega imk)=\frac{\sin(\omega mn)}{\sin(\omega m)}\times \exp[\omega (n+1)im].
\end{equation*}
Putting $\alpha_k=\exp(2\omega imk)$, we get the desired result (\ref{4.12}).

\end{proof}
The following approximation result for the Mellin transform (\ref{4.5}) in semi-inner product $C^*$-modules holds, (see \cite[Theorem 4]{dra1}).

\begin{proposition}
Let $X$ be a semi-inner product $C^*$-module, $a\in X$. If $x=(x_1,...,x_n)\in X^n, r\geq 0$ are such that
\begin{equation*}
\|x_i-a\|\leq r,\text{ for all i }\in \{1, . . . , n\},
\end{equation*}
then
\begin{multline}
\Big\|\mathcal{M}(x)(m)-S_{m-1}(n).\frac{1}{n}\sum_{k=1}^n x_k\Big\|\label{4.15}\\
\leq r\big[nS_{2m-2}(n)-S_{m-1}^2(n)\big]^\frac{1}{2}, m\in\{1,...,n\},
\end{multline}
where $S_p(n), p\in \mathbb{R}, n\in \mathbb{N}$ is the $p$-powered sum of the first $n$ natural
numbers, i.e.,
\begin{equation*}
S_p(n):=\sum_{k=1}^n k^p.
\end{equation*}
\end{proposition}

Consider the following particular values of Mellin Transform
\[
\mu_1(x):=\sum_{k=1}^n kx_k
\]
and
\[
\mu_2(x):=\sum_{k=1}^n k^2x_k.
\]
The following Corollary is a generalization of \cite[Corollary 4]{dra1}, furthermore, the quantities in the right hand sides of inequalities (5.5) and (5.6) in \cite[Corollary 4]{dra1} have been corrected by the following inequalities (\ref{4.16}) and (\ref{4.17}) respectively.

\begin{corollary}
Let $X$ be a semi-inner product $C^*$-module, $a\in X$. If $x=(x_1,...,x_n)\in X^n, r\geq 0$ are such that
\begin{equation*}
\|x_i-a\|\leq r,\text{ for all i }\in \{1, . . . , n\},
\end{equation*}
then
\begin{equation}
\left\|\mu_1(x)-\frac{n+1}{2}.\sum_{k=1}^n x_k\right\|\leq \frac{rn}{2} \left[\frac{(n-1)(n+1)}{3}\right]^{\frac{1}{2}},\label{4.16}
\end{equation}
and
\begin{multline}
\Big\|\mu_2(x)-\frac{(n+1)(2n+1)}{6}.\sum_{k=1}^n x_k\Big\|\\
\leq\frac{rn}{6\sqrt{5}}\sqrt{(n-1)(n+1)(2n+1)(8n+11)}.\label{4.17}
\end{multline}

\end{corollary}

Other inequalities related to the Gr\"{u}ss type discrete inequalities for polynomials with coefficients in a Hilbert space
such as Theorem 61, Theorem 62, Corollary 52 in \cite{dra}, have
versions that are valid for polynomials with coefficients in a semi-inner $C^*$-module. However, the details are
omitted.

%---------------------------------------------------------------------------------------%

\section{Gr\"{u}ss type inequalities in semi-inner product $H^*$-modules}

The following Theorem is a version of Corollary \ref{c3.1} for $H^\ast $-modules.
\begin{theorem}
\label{t5.1} Let $X$ be a semi-inner product $H^\ast $-module, $a,b\in X$ and $\overline{p}=(p_1,...,p_n)\in \mathbb{R}^n$ a probability vector. If $\overline{x}=(x_1,...,x_n), \overline{y}=(y_1,...,y_n)\in X^n, r\geq 0, s\geq 0$ are such that
\begin{equation}
\big\||x_i-a|\big\|\leq r,\quad \big\||y_i-b|\big\|\leq s,\text{ for all i }\in \{1, . . . , n\},\label{5.1}
\end{equation}
then the following inequality holds
\begin{equation}
\big|\tau(G_{\overline{p}}(\overline{x},\overline{y}))\big|\leq rs\label{5.2}
\end{equation}
The constant 1 coefficient of $rs$ in the inequalities (\ref{5.2}) is sharp.
\end{theorem}

\begin{proof}
By strong Schwarz inequality (\ref{2.3}) we have
\begin{equation}
\tau(G_{\overline{p}}(\overline{x},\overline{y}))^2\leq tr(G_{\overline{p}}(\overline{x}))tr(G_{\overline{p}}(\overline{y})).\label{5.3}
\end{equation}
 From inequality (\ref{2.9}) in Lemma \ref{l2.1} we obtain

\begin{equation}
tr(G_{\overline{p}}(\overline{x}))\leq \sum_{i=1}^n p_i\big\||x_i-a|\big\|^2\leq r^2,\label{5.4}
\end{equation}
similarly
\begin{equation}
tr(G_{\overline{p}}(\overline{y}))\leq \sum_{i=1}^n p_i\big\||y_i-b|\big\|^2\leq s^2\label{5.5}.
\end{equation}
Now (\ref{5.3}), (\ref{5.4}) and (\ref{5.5}) imply (\ref{5.2}).

The fact that the constant 1 is sharp may be proven in a similar manner to the one embodied in the proof of Corollary \ref{c3.1}.
 We omit the details.
\end{proof}
The following companion of the Gr\"{u}ss inequality for $H^\ast $-modules.

\begin{theorem}
\label{t5.2} Let $X$ be a semi-inner product $H^\ast $-module, $a,b\in X$
and $\overline{p}=(p_1,...,p_n)\in \mathbb{R}^n$ a probability vector. If $\overline{x}=(x_1,...,x_n), \overline{y}=(y_1,...,y_n)\in X^n, r\geq 0,s\geq 0$ are such that

\begin{equation}
\big\||x_i-a|\big\|\leq r, ~~\big\||y_i-b|\big\|\leq s \text{ for all } i\in \{1,...,n\}\label{5.6}
\end{equation}

then the following inequality holds
\begin{equation}
\big|\tau(G_{\overline{p}}(\overline{x},\overline{y}))\big|\leq rs-\left\|\Big|\sum_{i=1}^n p_i(x_i-a)\Big|\right\|\left\|\Big|\sum_{i=1}^n p_i(y_i-b)\Big|\right\|\leq rs.\label{5.7}
\end{equation}
\end{theorem}

\begin{proof}
From equality (\ref{2.8}) in Lemma \ref{l2.1} for $\overline{y}=\overline{x}$ and $b=a$, we get
\begin{equation}
tr(G_{\overline{p}}(\overline{x}))=\sum_{i=1}^n p_i\big\||x_i -a|\big\|^2-
\left\|\Big|\sum_{i=1}^n p_i (x_i -a)\Big|\right\|^2\label{5.8}.
\end{equation}
Similarly for every $b\in X$ by substitution $\overline{x}$ with $\overline{y}$, and $a$ with $b$, we have
\begin{equation}
tr(G_{\overline{p}}(\overline{y}))=\sum_{i=1}^n p_i\big\||y_i -b|\big\|^2-
\left\|\Big|\sum_{i=1}^n p_i (y_i -b)\Big|\right\|^2\label{5.9}.
\end{equation}
By strong Schwarz inequality (\ref{2.3}) we have
\begin{multline}
\big|\tau(G_{\overline{p}}(\overline{x},\overline{y}))\big|\leq\left[\sum_{i=1}^n p_i\big\||x_i -a|\big\|^2-
\left\|\Big|\sum_{i=1}^n p_i (x_i -a)\Big|\right\|^2\right]^\frac{1}{2}\label{5.10}\\
\times\left[\sum_{i=1}^n p_i\big\||y_i -b|\big\|^2-
\left\|\Big|\sum_{i=1}^n p_i (y_i -b)\Big|\right\|^2\right]^\frac{1}{2}\\
\leq\left[r^2-\left\|\Big|\sum_{i=1}^n p_i (x_i -a)\Big|\right\|^2\right]^\frac{1}{2}\left[s^2-\left\|\Big|\sum_{i=1}^n p_i (y_i -b)
\Big|\right\|^2\right]^\frac{1}{2}.
\end{multline}
Now using the elementary inequality for real numbers
\begin{equation}
(m^2-n^2)(p^2-q^2)\leq(mp-nq)^2\label{5.11}
\end{equation}
on
\begin{align*}
m&=r, &&n=\left\|\Big|\sum_{i=1}^n p_i (x_i -a)\Big|\right\|,\\
p&=s, &&q=\left\|\Big|\sum_{i=1}^n p_i (y_i -b)\Big|\right\|,
\end{align*}
we get the inequality (\ref{5.7}).
\end{proof}

\begin{corollary}
\label{c5.1} Let $X$ be a semi-inner product $H^\ast $-module, $a\in X$
and $\overline{p}=(p_1,...,p_n)\in \mathbb{R}^n$ a probability vector. If $\overline{x}=(x_1,...,x_n), \overline{y}=(y_1,...,y_n)\in X^n, r\geq 0$ are such that
\begin{equation}
\big\||x_i-a|\big\|\leq r \text{ for all } i\in \{1,...,n\}\label{5.12}
\end{equation}
then the following inequality holds
\begin{equation}
\big|\tau(G_{\overline{p}}(\overline{x},\overline{y}))\big|\leq r\left[ \sum_{i=1}^n p_i\big\||y_i|\big\|^2
-\left\|\Big|\sum_{i=1}^n p_iy_i\Big|\right\|^2\right]^{\frac{1}{2}}.\label{5.13}
\end{equation}
\end{corollary}

\begin{proof}
From the equality (\ref{5.8}) and the condition (\ref{5.12}) we have
\begin{align}\label{5.14}
tr(G_{\overline{p}}(\overline{x}))&=\sum_{i=1}^n p_i\big\||x_i -a|\big\|^2-
\left\|\Big|\sum_{i=1}^n p_i (x_i -a)\Big|\right\|^2\\
&\leq \sum_{i=1}^n p_i r^2-\left\|\Big|\sum_{i=1}^n p_i (x_i -a)\Big|\right\|^2\leq r^2.\notag
\end{align}

Using the inequalities (\ref{5.3}), (\ref{5.14}) and the equality (\ref{5.9}) we get
\begin{equation}
\big|\tau(G_{\overline{p}}(\overline{x},\overline{y}))\big|\leq r\left[ \sum_{i=1}^n p_i\big\||y_i-b|\big\|^2
-\left\|\Big|\sum_{i=1}^n p_i(y_i-b)\Big|\right\|^2\right]^{\frac{1}{2}},\label{5.15}
\end{equation}
and for $b=0$ we get the inequality (\ref{5.13}).
\end{proof}
There exist a version of Remark \ref{r3.1} for semi-inner product $H^*$-modules and
 there are applications from theorems and results in this section for the approximation of
some discrete transforms in a semi-inner product $H^*$-module. However, the details are omitted but each of them can be proven
 in a similar manner as section 4.

Let $\mathcal{A}$ be a Banach $*$-algebra and $X$ be a semi-inner product $\mathcal{A}$-module (see \cite{gha}).
Utilizing Schwarz inequality (\ref{2.5}) or (\ref{2.6}) and a version of Lemma \ref{l2.1} for semi-inner product $\mathcal{A}$-module $X$.
The technique of the proof of Theorem \ref{t3.1} is applicable to semi-inner product
Banach $*$-modules as well.

%------------------------------------------------------------------------------%
%End of journal.tex
%------------------------------------------------------------------------------%

\end{document}